\documentclass[leqno,11pt]{article}

\usepackage{amssymb,mathtools,amsthm}
\usepackage[latin1]{inputenc}
\usepackage{graphicx,color}
\usepackage[dvipsnames]{xcolor}
\usepackage{stackrel}
\usepackage[toc,page]{appendix}
\usepackage{geometry} 
\usepackage[english]{babel}
\usepackage[colorlinks, citecolor=blue]{hyperref}
\usepackage{listings}

\usepackage{tikz} 
\usepackage{pgfplots}
\usepgfplotslibrary{groupplots}
\usepgflibrary{arrows}
\usetikzlibrary{calc}

\newcommand{\R}{\mathbb R}
\newcommand{\N}{\mathbb N}
\newcommand{\Z}{\mathbb Z}

\newcommand{\al}{\alpha}

\newcommand{\ve}{\varepsilon}
\newcommand{\vt}{\vartheta}
\newcommand{\vp}{\varphi}

\newcommand{\sset}{\subseteq} 
\newcommand{\uguale}{\stackrel{.}{=}} 

\theoremstyle{plain}
\newtheorem{theorem}{Theorem}[section]
\newtheorem{lemma}[theorem]{Lemma}
\newtheorem{proposition}[theorem]{Proposition}

\theoremstyle{definition}
\newtheorem{example}[theorem]{Example}
\newtheorem{definition}[theorem]{Definition}
\newtheorem{remark}[theorem]{Remark}

\definecolor{ceruleanblue}{rgb}{0.16, 0.32, 0.75}
\definecolor{cocoabrown}{rgb}{0.82, 0.41, 0.12}
\definecolor{crimsonglory}{rgb}{0.75, 0.0, 0.2}
\definecolor{aqua}{rgb}{0, 0.51, 0.5}

\definecolor{codegreen}{rgb}{0,0.6,0}
\definecolor{codegray}{rgb}{0.5,0.5,0.5}
\definecolor{codepurple}{rgb}{0.58,0,0.82}
\definecolor{backcolour}{rgb}{0.95,0.95,0.95}

\lstdefinelanguage{TOML}{
    comment = [l]{\#},
    keywords = {true, false},
    morestring = [b]{"}
}

\lstdefinelanguage{Julia}
  {morekeywords={abstract,break,case,catch,const,continue,do,else,elseif,%
      end,export,false,for,function,immutable,import,importall,if,in,%
      macro,module,otherwise,quote,return,switch,true,try,type,typealias,%
      using,while},
   sensitive=true,
   morecomment=[l]\#,
   morecomment=[n]{\#=}{=\#},
   morestring=[s]{"}{"},
   morestring=[m]{'}{'},
}[keywords,comments,strings]

\lstdefinestyle{mystyle}{
	backgroundcolor=\color{backcolour},
    commentstyle=\color{codegreen},
    keywordstyle=\color{magenta},
    numberstyle=\tiny\color{codegray},
    stringstyle=\color{codepurple},
    basicstyle=\ttfamily\footnotesize,
    breakatwhitespace=false,         
    breaklines=true,                 
    captionpos=b,                    
    keepspaces=true,                 
    numbers=none,                    
    numbersep=5pt,                  
    showspaces=false,                
    showstringspaces=false,
    showtabs=false,                  
    tabsize=2,
	xleftmargin=0.2cm,
}

\usepackage{physics}
\let\grad\relax
\DeclareMathOperator{\grad}{grad}
\DeclareMathOperator{\hess}{Hess}

\title{Equivariant optimisation for the gravitational $n$-body problem: a computational factory of symmetric orbits}
\author{Vivina Barutello \and Mattia G. Bergomi \and Gian Marco Canneori
		\and Roberto Ciccarelli \and Davide L. Ferrario \and Susanna Terracini 
		\and Pietro Vertechi}

\date{}

\begin{document}

\maketitle

\abstract{In this paper we present \emph{SymOrb.jl}} \cite{SymOrbJL}, a software which combines group representation theory and variational methods to provide numerical solutions of singular dynamical systems of paramount relevance in Celestial Mechanics and other interacting particles models. Among all, it prepares for large-scale search of symmetric periodic orbits for the classical $n$-body problem and their classification, paving the way towards a computational validation of Poincar\'e conjecture about the density of periodic orbits. Through the accessible language of Julia, Symorb.jl offers a unified implementation of an earlier version  \cite{symorb}. This paper provides theoretical and practical guidelines for the specific approach we adopt, complemented with examples. \vspace{2mm}

\tableofcontents
\section*{Introduction}

Among natural dynamical systems arising in classical mechanics, the gravitational $n$-body problem has become a paradigmatic example of a singular system, both for its simple and clean formulation, and the notable complexity of its orbits' structure. The question is the following: what are the trajectories of $n$ heavy bodies that move in the Euclidean space under their mutual gravitational attractions, and how do we describe their motions? For $x_1,\ldots,x_n\in\R^3$ heavy bodies with masses $m_1,\ldots,m_n>0$, classical Newton laws provide the equations of motion, which
\[
	m_k\ddot{x}_k(t)=-\sum\limits_{j\neq k}m_km_j\frac{x_k(t)-x_j(t)}{\|x_k(t)-x_j(t)\|^3},\quad \text{for each}\ k\in\{1,\ldots,n\},
\]
where $\|\cdot\|$ is the Euclidean norm in $\R^3$, so that a solution of the $n$-body problem is the list of trajectories of the bodies. At first glance, the set of singularities of this system of equations might be large and difficult to manage: whenever a body $x_j$ collapses on another body $x_i$, the motion equations are not defined. When $n=2$, the picture is quite clear: each body follows a conic section, which can be easily classified in terms of the mechanical energy. Nonetheless, the 3-body problem is already a non-integrable system, and small perturbations on initial conditions can cause a dramatic displacement of nearby orbits. This sensitivity of the initial data flow has even caught the attention of contemporary folklore, raising the interest of this elegant old problem.  The theoretical foundations which motivate the research fervour around this problem can be traced back to the famous Poincar\'e conjecture \cite{Poincare}. A possibly chaotic dynamics is strictly related to the existence of a dense family of periodic solutions:

\medskip

{\it
``D'ailleurs, ce qui nous
rend ces solutions p\'{e}rio\-di\-ques si
pr\'{e}cieuses, c'est qu'elles sont, pour ainsi dire,
la seule br\`{e}che par o\`{u} nous puissons
essayer de p\'{e}n\'{e}trer dans une place
jusqu'ici r\'{e}put\'{e}e inabordable...''.
}

\medskip

Indeed, according to Poincar\'e, all trajectories can be approximated by periodic ones:

\medskip

{\it ``...voici un fait que je n'ai
pu d\'emontrer rigoureusement,
mais qui me parait pourtant tr\`es vraisemblable.
\'Etant donn\'ees des \'equations de la forme d\'efinie
dans
le n. 13\footnote{Formula 13, mentioned by Poincar\'{e}, is the Hamilton equation.}
et une solution particuli\`ere quelconque de ces \'equa\-tions,
one peut toujours trouver une solution
p\'erio\-dique (dont la p\'eriode
peut, il est vrai, \^etre tr\`es longue),
telle que la diff\'erence entre
les deux solutions soit
aussi petite qu'on le veut, pendant un temps aussi
long qu'on le veut.''
}

\medskip

As a result, the presence of rich sets of periodic solutions is a suitable indicator of complexity for dynamical systems. Moreover, the knowledge of classical periodic orbits has profound implications for the associated quantum problem. Periodic orbits form the backbone of semiclassical approximations, such as the Gutzwiller trace formula, which connects the quantum energy levels with the classical periodic trajectories. This connection is pivotal for interpreting spectral properties and predicting quantum states.

Since the 90s of the last century, variational methods have shown to be successful in producing selected solutions of singular systems (among others, see \cite{BaRa1989,BaRa1991,SerraTer1992, MajTer1993,SerraTer94, MajTer1995}) and eventually such techniques have found a powerful application in search of periodic orbits for the $n$-body problem (see in particular the celebrated papers \cite{ChMont1999,FT2003}). Nowadays, it is widely known that one could look for collision-less periodic solutions among critical points of the associated Lagrange-action functional
\begin{equation}
	\mathcal{A}(x)=\int_0^T\left[\frac12\sum\limits_{i=1}^n m_i\|\dot{x}(t)\|^2+\sum\limits_{i<j}\frac{m_im_j}{\|x_i(t)-x_j(t)\|}\right]\,dt.
	\label{eq:action}
\end{equation}
over the space of T-periodic loops: a peculiar guide on this approach can be found in the monograph \cite{AC-Z}. Within the class of all critical points, it seems more natural to look for minimisers of $\mathcal{A}$, especially if it is feasible to recover coercivity and to work with \emph{closed} subsets of the paths space. 

To a wider extent, an essential guide in the hunting of solutions via variational methods can be summarised in these steps: 
\begin{itemize}
	\item introduce a functional which is differentiable over a space of candidate solutions. Such a functional can have a geometrical or dynamical connection with the problem and naturally arises from the weak formulation of the motion equations. This naturally leads to work with the Sobolev space of $H^1$ functions (see e.g. \cite{struwe}). Usually, Jacobi-length and Maupertuis functionals reflect the Riemannian structure of the problem, while Energy and Lagrange-action functionals are more related to the dynamical nature of the problem;
	\item state and prove a variational principle which guarantees that non-trivial collision-less critical points of this functional, or at least some re-parametrisations of them, solve the motion equations. Depending on the functional involved, one can have a \emph{Maupertuis principle} or a \emph{Least action principle};
    \item look for critical points of the chosen functional. For instance, if coercivity and compactness of sub-levels for the functional can be proved, direct methods in calculus of variations apply (e.g., see again \cite{struwe}) and one naturally looks for minimisers. 
\end{itemize}
Concerning the $n$-body problem, the first obstruction is the lack of compactness in the open set of collision-less paths and hence of the sub-levels of $\mathcal{A}$. Furthermore, an elementary remark excludes coercivity of the action functional: any sequence of constant loops with diverging mutual distances between the bodies has a diverging $H^1$-norm, but $\mathcal{A}$ remains bounded on such sequence. Nonetheless, different types of reduction of the $n$-body problem to more affordable models allow to recover coercivity and to succeed in the minimisation process (for the symmetric $n$-body problem see \cite{BeCZ1991,ChMont1999,FT2003,Chen2001,RT95,ChenVen2000,BT2004,FuGroNeg,MonSte2013,Simo2001}. We also refer to the $n$-centre problem \cite{Bol1984,KK1992,Kn2002,ST2012,BolKoz2017,BCTarma,BC2024}). This paper is concerned with a particular reduction case, requiring the candidate paths to satisfy some chosen symmetries in time, space, and body labels. 

In this direction, a very successful approach has been presented in \cite{FT2003}, where the authors recover coercivity by symmetrising the space of paths through the action of a finite group $G$. Under major hypotheses on the action of $G$, the authors have shown that minimisers of $\mathcal{A}$ are collision-less and so they actually solve a \emph{symmetrical} $n$-body problem (see also \cite{Fer2024,Fer2006,Fer2007,BFT1,XiaZhou2022}). This rigorous argument has been supported by a numerical implementation of a constrained optimisation algorithm, which allows us to find a finite group $G$ verifying some chosen symmetries and to check if a $G$-equivariant optimisation is achievable. Once a $G$-equivariant space of closed paths is generated, an optimisation routine provides a candidate numerical solution of the $n$-body problem, which displays the symmetries fixed by $G$.

The first version of this algorithm, named \emph{Symorb} by its creator Davide Ferrario, is available in \cite{symorb}. Symorb is based on a combination of Python, Fortran, and the computer algebra symbolic language GAP (see \cite{GAP}). The present paper introduces \emph{SymOrb.jl} \cite{SymOrbJL}, a reimplementation of Symorb's Fortran and Python routines in a unified language, Julia~\cite{bezanson2017julia}. It offers theoretical and practical guidelines for the variational approach we are adopting. We provide instructions for generating symmetric, periodic solutions to the $n$-body problem and subsequently computing them numerically.

\paragraph{Conclusions and perspectives}
The variational-computational approach to the search for symmetric solutions in the gravitational $n$-body problem represents a novel and flexible methodology. Through a combination of analytical-functional methods, it stands out from past research based on solving initial data problems. It offers unprecedented efficiency in identifying stable and unstable periodic orbits and understanding the intricate dynamics of multi-body celestial systems. It can easily be used for large-scale searches of equivariant trajectories, which we intend to classify with respect to their symmetry properties, variational characterisation and stability. We can collect symmetric orbits in an accessible database and we plan to use them as building blocks in the search of complex periodic trajectories, which do not necessarily satisfy a symmetry constraint. Furthermore, our computational exploration can be complemented with computer-assisted proofs, for rigorously proving the existence of new periodic solutions (see \cite{KapZgl, AriBarTer, KapSimo, CalAzpLesMir, CapKepMir}).

On the computational side, Julia provides first-class support for linear and nonlinear operations on mutidimensional arrays with an accessible, concise syntax. Julia's Just In Time (JIT) compilation model enables execution speed on par with statically compiled languages while allowing for an efficient and interactive developer experience.
This is particularly crucial in our scenario, as \texttt{SymOrb.jl} needs to run complex custom code in the hot loop of an optimization routine, thus not only library but also user-defined code needs to be highly performant. Additionally, the Julia ecosystem encompasses support for GPU programming~\cite{besard2019prototyping} and modern optimization methods \cite{Optim.jl}.

\paragraph{Notations}
Let $\mathbb{R}^d$ be the $d$-dimensional Euclidean space and $m_1,\ldots,m_n>0$ be positive masses, with $n\geq 2$. We introduce the configuration space of $n$ heavy particles with centre of mass in the origin:
\[
\mathcal{X}\uguale\left\lbrace x=(x_1,\ldots,x_n)\in(\mathbb{R}^d)^n\,\middle\vert\ \sum\limits_{i=1}^nm_ix_i=0 \right\rbrace.
\]
Concerning collisions, we define the following singular sets 
\[
\Delta_{ij}\uguale\left\lbrace x\in\mathcal{X}\,\middle\vert\,x_i=x_j\right\rbrace,\quad\Delta\uguale\bigcup\limits_{i,j}\Delta_{ij}
\]
so that we can also introduce the set of collision-free configurations $\hat{\mathcal{X}}\uguale\mathcal{X}\setminus\Delta$. The potential function associated with the classical $n$-body problem is usually introduced as 
\[
U(x)\uguale\sum\limits_{i<j}\frac{m_im_j}{\|x_i-x_j\|},\quad \text{for}\ x\in\mathcal{X},
\]
while the kinetic energy function is
\[
K(\dot{x})\uguale\frac12\sum\limits_{i=1}^nm_i\|\dot{x}_i\|^2,\quad\text{for}\ \dot{x}\in T_x\mathcal{X};
\]
finally, the Lagrangian function reads
\[
L(x,\dot{x})\uguale K(\dot{x})+U(x),\quad\text{for}\ x\in\mathcal{X},\ \dot{x}\in T_x\mathcal{X}.
\]
With these notations, the motion equations of $n$ bodies $x_1(t),\ldots,x_n(t)$ in the configuration space $\mathcal{X}$ read:
\begin{equation}\label{eq:newton}
	m_i\ddot{x}_i=\frac{\partial U}{\partial x_i}, \qquad i=1,\ldots,n,
\end{equation}
which are nothing but the Euler-Lagrange equations of the above Lagrange function $L$. 
\begin{remark}\label{rem:general_problem}
	The theoretical results recalled in this paper and proved in \cite{FT2003}, together with the numerical procedures summarised in the last section, adapt to a wider class of potential functions $U$. In particular, we refer to the minimal hypotheses required on $U$ introduced in \cite{BFT2}, in order to guarantee the existence of collision-less critical points of $\mathcal{A}$. For instance, the $\al$-homogenous $n$-body problem can be considered, but also quasi-homogenous and logaritmic potentials are allowed in this analysis (for completeness, we particularly refer to \cite[Section 7]{BFT2}). As will be clear in Section \ref{sec:optimisation}, the numerical optimisation routine allows us to choose a generic function $f=f(\|x_i-x_j\|)$ instead of the singular quotient $\frac{1}{\|x_i-x_j\|}$. Clearly, the convergence to a collision-less critical point is not always guaranteed. 
	
	When dealing with the $n$-body problem, it is natural to rephrase this dynamical system in a uniform rotating frame. As a matter of fact, the results proved in \cite{FT2003} are valid in this more general situation, where the kinetic energy of the system clearly depends on $x$ (see also Section \ref{sec:optimisation}). For the sake of simplicity, we prefer to present the theoretical statements of this paper in the interial frame. 
\end{remark}

As already remarked in the Introduction, there is a major interest in looking for periodic solutions of \eqref{eq:newton}. A periodic solution of \eqref{eq:newton} has the property that every trajectory $x_i(t)$ of the $i$-th body is a closed curve in $\R^d$, usually referred as a \emph{loop} in the configuration space. As a short-hand notation, for $T>0$ we consider the time circle $\mathbb{T}\simeq\mathbb{R}/T\mathbb{Z}$, so that the space of $H^1$-periodic loops in $\mathcal{X}$ can be defined as
\[
\Lambda\uguale H^1(\mathbb{T};\mathcal{X})=\{x\in L^2([0,T];\mathcal{X}):\,\dot{x}\in L^2([0,T]; \mathcal{X}),\ x(0)=x(T)\},
\]
where $\dot{x}$ denotes the weak derivative of $x$ and the boundary condition makes sense since $H^1$ functions are indeed continuous by classical Sobolev embedding. Recalling that $\hat{\mathcal{X}}$ is the set of non-collision configurations, we introduce also the non-collision counterpart of $\Lambda$
\[
	\hat{\Lambda}\uguale H^1(\mathbb{T};\hat{\mathcal{X}})
\] 
and, for $x=x(t)\in\Lambda$, we define the set of \emph{collision times} of $x$ as the set $x^{-1}\Delta\sset\mathbb{T}$. We finally consider the Lagrange action functional $\mathcal{A}\colon\Lambda\to\R\cup\{+\infty\}$ defined as
\[
\mathcal{A}(x)\uguale\int_0^TL(x(t),\dot{x}(t))\,dt,\quad\text{for}\ x\in\Lambda.
\]
Note that, since $\mathcal{A}\in\mathcal{C}^1(\hat{\Lambda})$, it is well-known that if $x\in\hat{\Lambda}$ is a critical point of $\mathcal{A}$, then $x$ is a $T$-periodic solution of \eqref{eq:newton} (see for instance Lemma 19.1 in \cite{AC-Z}).

\paragraph*{Outline of the paper}
The paper is organised as follows. In Section \ref{sec:sym_group} we briefly recall the concept of \emph{group action} and its connection with symmetries. In Section \ref{sec:sym_lambda} we will see how the machinery of group actions can be very usueful to symmetrise the loop space $\Lambda$, we will state some ad-hoc variational principles for equivariant loops and we will collect some useful properties and classifications of groups acting on $\Lambda$. Finally, Section \ref{sec:optimisation} is devoted to the search of equivariant critical points of $\mathcal{A}$, together with a survey on the numerical implementation of this procedure. The appendix sections contains some supplementary resources on the topics of this paper.

\section{Symmetries as group actions}\label{sec:sym_group}

Representation theory offers the perspective of thinking of groups as structures which \emph{act} on mathematical objects, such as sets or spaces. Our first aim is to endow the space of $H^1$-periodic loops $\Lambda$ with a prescribed symmetry constraint. A powerful way to describe such a symmetry is through the \emph{action} of a finite group $G$ on $\Lambda$. A rigorous definition of \emph{group action} is given in the following:

\begin{definition}[Group action]\label{def:action}
	We say that a finite group $G$ acts (on the left) on a set $X$ if there exists a map $\phi\colon G\times X\to X$ such that
	\begin{itemize}
		\item $\phi(1,x)=x$ for any $x\in X$;
		\item $\phi(g,\phi(h,x))=\phi(gh,x)$ for any $g,h\in G$ and $x\in X$,
	\end{itemize}
i.e., the map $\phi$ commutes with the group operation. We will often write $gx$ instead of $\phi(g,x)$. We say that the action of $G$ on a set $X$ is \emph{transitive} if for any $x,y\in X$ there exists $g\in G$ such that $gx=y$.
\end{definition}

Therefore, a group action is essentially a way for $G$ to interact with the set $X$. Given $x\in X$, it could be useful to regroup those elements of $G$ which fix $x$; on the other hand, we could need to gather those elements of $X$ which are invariant through the action of the whole $G$ or of a subgroup $H$ of $G$. For instance, if $X=\R^2$ and $g\in G$ is a reflection with respect to the line $y=x$, we see that all the points $(x,x)$ are fixed by $g$. In the following, we will use the notations $H<G$ and $H\lhd G$ to denote respectively the subgroup and normal subgroup relations.
\begin{definition}[Isotropy groups and projectors]\label{def:isotropy}
	For an element $x\in X$ we define the \emph{isotropy group} of $x$ in $G$ as
	\[
	G_x\uguale\{g\in G:\, gx=x\},
	\]
	and $G_x < G$. Moreover, if $H<G$, we define the set of \emph{points fixed by} $H$ as the set
	\[
	X^H\uguale\{x\in X:\,G_x\supseteq H\}=\{x\in X:\, hx=x,\ \forall\,h\in H\}.
	\]
	We define also the \emph{projector onto} $X^H$ as the map $\pi_H\colon X\to X^H $ which acts in this way:
	\begin{equation}\label{eq:projector}
		\pi_H\,x\uguale \frac{1}{|H|}\sum\limits_{h\in H} hx.
	\end{equation}
\end{definition}

\begin{remark}\label{rem:weyl}
	We recall that if $H<G$, then the normaliser $N_GH$ of $H$ in $G$ is the smaller subgroup of $G$ in which $H$ is normal. Moreover, the \emph{Weyl group} of $H$ in $G$ is the quotient group $W_GH = N_GH/H$. If $G$ acts on $X$, then the Weyl group $W_GH$ acts on $X^H$. Indeed, since  and $H\lhd N_GH$, then the cosets of $H$ naturally act on $X^H$.
\end{remark} 

From a geometrical point of view groups can be considered as \emph{sets of symmetries} of a mathematical object. For instance, the symmetries of a square are exactly the elements of the dihedral group $D_8$ (see Definition \ref{def:dihedral}). 


In this context, given a set $X$ and a group $G$ acting on it, the action of $G$ on $X$ induces a group homomorphism between $G$ and $\Sigma_X$ (the group of all permutations of elements of $X$)
\[
a\colon G \to \Sigma_X,\quad a(g)\colon X\to X,\quad a(g)x=gx. 
\]
Indeed, for any $g \in G$, the group action implies that the map $a(g)$ is a bijection with inverse $a(g^{-1})$; for the same reason, the set $\left\{a(g): g \in G\right\}$ is a subgroup of $\Sigma_X$.
The homomorphism $a$ is an example of \emph{representation} of $G$, which is obtained by starting from its action on $X$. Below, we list some examples of representations of $G$.

\begin{example}[$X$ finite, action of $D_{2n}$]\label{ex:finite}
Take $X=\{1,\ldots,n\}$ so that $\Sigma_X=\Sigma_n$, the \emph{symmetric group} of $n$ elements. We can label the vertices of a regular $n$-gon with elements of $X$ and take $G=D_{2n}$. If $g$ is the rotation of $\frac{2\pi}{n}$, then $g$ corresponds to the cycle $(1,\ldots,n)$. On the other hand, if $n$ is even and $g$ is the reflection with respect to the symmetry axis through the vertices 1 and $n/2+1$, then $g$ is represented in $\Sigma_n$ as the permutation $(2,n)(3,n-1)\cdots(n/2,n/2+2)$. If $n$ is odd, the reflection is represented as the permutation $(2,n)(3,n-1)\cdots((n+1)/2,(n+3)/2)$. We then understand that the way a group $G$ acts on a space determines a representation of $G$, and viceversa.
\end{example}

\begin{example}[$X$ linear space]\label{ex:linear}
	Assume now that $X$ is a linear space over $\mathbb R$. It is natural to require that the action of $G$ on $X$ preserves the linear structure of the vector space, that is, for any $g \in G$, $x_1,x_2 \in X$ and $\lambda,\mu \in \mathbb R$
	\[
	a(g)(\lambda x_1+\mu x_2) = \lambda a(g)x_1+\mu a(g)x_2.
	\]
	This means that, for any $g \in G$, the map $a(g)$ is a linear transformation from $X$ to itself and thus $a(g) \in GL(X)$. In this case we refer to $a\colon G \to GL(X)$ as a \emph{linear representation of $G$}. On the other hand, when $X={\mathbb R}^n$ then the whole $GL(n)$ acts on ${\mathbb R}^n$ in the natural way. A group $G$ acting on  ${\mathbb R}^n$ can be represented as a subgroup of $GL(n)$; in particular, one could require not only the linearity of the group action, but also to be angle-preserving. In the latter case, $G$ is represented into $O(n)$ and the representation is usually called \emph{orthogonal}. Note that such representation always exists (the trivial action is always admitted).
\end{example}

\section{Action of finite groups on $\Lambda$}\label{sec:sym_lambda}

In this section we deepen the main ideas of \cite{FT2003} and we define group actions on the loop space $\Lambda = H^1(\mathbb T; \mathcal X)$ in order to introduce a symmetry constraint. As a preliminary remark, since $\Lambda$ is infinite-dimensional, it is clear that we have to look for finite dimensional objects on which a group can effectively act. 

Recall that $\mathbb T = \mathbb R /{T\mathbb Z}$ and that, for any $x=x(t) \in \Lambda$, we can write 
\[
x(t) = (x_1(t),\ldots,x_n(t)) \in \mathcal X\sset(\mathbb R^d)^n, \quad \text{for any } t \in \mathbb T.
\]
We can identify three finite dimensional objects in $\Lambda$: 
\begin{itemize}
	\item the {\em space} $\mathbb R^d$ in which every component $x_i(t)$ lies; 
	\item the {\em time circle} $\mathbb T\subset \mathbb R^2$ which represents the period of a trajectory;
	\item the {\em index set} $\{1,\ldots,n\}$ on which the $n$-bodies are labelled.
\end{itemize}
A way to define the action of $G$ on $\Lambda$ is then through three different  representations of $G$ as a group action respectively on $\mathbb R^d$, $\mathbb R^2$ and $\{1,\ldots,n\}$. 
More precisely, $G$ is represented as a subgroup of $O(d)$, of $O(2)$ and of the symmetric group $\Sigma_n$ (see Examples \ref{ex:finite} and \ref{ex:linear}) via the following homomorphisms
\[
\rho : G \to O(d), \quad \tau : G \to O(2), \quad \sigma: G \to \Sigma_n. 
\]
As an intermediate step, we use $\rho$ and $\sigma$ to define the action of $G$ on $\mathcal X$
\[
\phi: G \times \mathcal X \to \mathcal X, \quad \phi(g,x) = gx,
\]
where 
\[
gx=g(x_1,\ldots,x_n) = \left(\rho(g) x_{\sigma(g^{-1})1},\ldots,\rho(g) x_{\sigma(g^{-1})n}\right),
\]
or, shortly
\[
(gx)_i=g x_{g^{-1}i}.
\]
Let us observe that this action is orthogonal, in the sense that it preserves the inner product of $(\R^d)^n$; this is easy to verify since $\rho(g)\in O(d)$ and the permutation $\sigma(g^{-1})$ does not affect the inner product. For reasons that will be clear later, it makes sense to require that $\sigma\colon G\to\Sigma_n$ satisfies the following property
\begin{equation}\label{property:sigma}
	\forall\,g\in G,\ \text{if}\ \sigma(g^{-1})i=j\ \text{for some}\ i\neq j,\ \text{then}\ m_i=m_j,
\end{equation}
meaning that only bodies with equal masses can be \emph{interchanged} by $\sigma$. For instance, if the action of $G$ is transitive, then we will work with $n$ equal masses.

Since the time circle $\mathbb{T}$ is a 1-dimensional object immersed in $\R^2$, the action of $G$ on $\mathbb{T}$ can be defined through the representation $\tau\colon G\to O(2)$ in this way
\[
gt=\tau(g^{-1})t,\ \mbox{for}\ t\in\mathbb{T},\ g\in G.
\]

\begin{remark}\label{rem:tau_O(2)}
	Since $\tau\colon G\to O(2)$ is a homomorphism, then:
	\begin{itemize}
		\item $\ker\tau\lhd G$;
		\item $G/\ker\tau\simeq\tau(G)$ and so it is a finite subgroup of $O(2)$.
	\end{itemize}
    Thanks to Proposition \ref{prop:subgroups_orthogonal}, we deduce that $\tau(G)$ is either a cyclic or dihedral subgroup of $O(2)$. Since $\tau$ models an action of $G$ on $\mathbb{T}$, then any $\tau(g)$ is either a rotation or a reflection in the plane. With a slight abuse of notations, or assuming that $\theta$ and $t$ belog to the same circle, for any $t\in\mathbb{T}$ we have:
    \begin{itemize}
    	\item if $\tau(g)$ is a counter-clockwise rotation of $\vt$, then:
    	\[
    	\tau(g)t=\vt+t;
    	\]
    	\item if $\tau(g)$ is a reflection with respect to a line through the origin which forms an angle of $\vt/2$ with the horizontal axis, then:
    	\[
    	\tau(g)t=\vt-t.
    	\]
    \end{itemize}
\end{remark}
We have finally obtained a way to represent the action of $G$ on the loop space $\Lambda$ by means of the homomorphisms $\rho$, $\sigma$ and $\tau$, which is the following
\[
(gx)(t)=(\rho(g)x_{\sigma(g^{-1})1}(\tau(g^{-1})t),\ldots,\rho(g)x_{\sigma(g^{-1})n}(\tau(g^{-1})t))
\]
which in short we will write as 
\[
(gx)(t)=(gx_{g^{-1}i}(g^{-1}t))_i,
\]
for any $g\in G$, $t\in\mathbb{T}$ and $x\in\Lambda$. We also define the set of loops in $\Lambda$ fixed by $G$, or $G$-\emph{equivariant loops}, as the set
\[
\Lambda^G\uguale\{x\in\Lambda\colon (gx)(t)=x(t),\ \forall\,t\in\mathbb{T},\ g\in G\}
\]
and its collision-less subset
\[
\hat{\Lambda}^G\uguale\{x\in\hat{\Lambda}\colon (gx)(t)=x(t),\ \forall\,t\in\mathbb{T},\ g\in G\}.
\]
\begin{example}[$G=D_6$]\label{ex:d6_1}
	We consider the dihedral group 
	\[
	D_6 = \langle s,r : s^2 = r^3 = (sr)^2 = 1 \rangle, 
	\]
	which is the group of symmetries of a regular triangle in the $(x,y)$-plane. For simplicity, assume that $T=2\pi$ and one of the vertices of the triangle lies on the line $y=x$. According to Definition \ref{def:dihedral}, we can define the representations $\tau, \rho$ and $\sigma$ through their values on the generators $s$ and $r$
	\[
	\tau(s) = \begin{pmatrix}
		0 & 1 \\ 1 & 0
	\end{pmatrix} ,
	\quad 
	\tau(r) = \begin{pmatrix}
	\cos \frac{2\pi}{3} & -\sin \frac{2\pi}{3} \\ \sin \frac{2\pi}{3} & \cos \frac{2\pi}{3}
	\end{pmatrix},
	\]
	\[
	\rho(s) = -\mathrm{Id}_2 ,
	\quad 
	\rho(r) = \mathrm{Id}_2 ,
	\]
	\[
	\sigma(s) = (1,2),
	\quad 
	\sigma(r) = (1,2,3).
	\]
	Note that, using these homomorphisms, we are modelling a 3-body problem in the plane. It is clear that these representations depends on the geometrical interpretation of the elements of $D_6$ as transformations which fix a regular triangle. Let us observe that 
	\[
	\ker (\tau) = \ker (\sigma) = 1, \quad 
	\ker (\rho) = \langle r \rangle,
	\]
	and also that, by assumption \eqref{property:sigma}, we have $m_1=m_2=m_3$. Let us remark that $x \in \Lambda^G$ if and only if for every $t\in\mathbb{T}$
	\[
	(sx)(t)=x(t) \quad \text{and} \quad (rx)(t)=x(t).
	\]	
	The first relation can be rewritten in this way
	\[
	\rho(s)
	x_{\sigma(s^{-1})j}(\tau(s^{-1})t) = x_j(t),\quad \forall\,j\in\{1,2,3\}
	\]
	and $\tau(s)$ is a reflection with respect to the line $y=x$. By means of the notations introduced in Remark \ref{rem:tau_O(2)}, $\tau(s)$ acts as the map $t\to \frac{2\pi}{4}-t$ and so the the equivariance reads:
	\[
	\begin{cases}
		x_1(t) = -x_2\left(\frac{2\pi}{4}-t\right) \\
		x_2(t) = -x_1\left(\frac{2\pi}{4}-t\right) \\
		x_3(t) = -x_3\left(\frac{2\pi}{4}-t\right).
	\end{cases}
	\]
	Now, $\rho(r)$ is the identity matrix and $\tau(r)$ is a counter-clockwise rotation of $\frac{2\pi}{3}$; therefore $\tau(r^{-1})$ is a clockwise rotation of the same angle, that in $\mathbb T$ corresponds to a time-shift of $-\frac{2\pi}{3}$. Hence the equivariance with respect to $r$ reads
	\[
	\begin{cases}
		x_1(t) = x_2\left(t-\frac{2\pi}{3}\right) \\
		x_2(t) = x_3\left(t-\frac{2\pi}{3}\right) \\
		x_3(t) = x_1\left(t-\frac{2\pi}{3}\right)
	\end{cases}
	\]
	which means that every $D_6$-equivariant loop is in particular a {\em choreography}.
\end{example}

\subsection{Equivariant variational principles}

The next result guarantees that $\Lambda^G$ is a loop space which fits a variational argument to find periodic solutions as critical points of the action functional. We will not give the details of the proof since they are quite standard and follow from the definition of group action on $\Lambda$. In particular, the equivariance of $\mathcal{A}$ follows from assumption \eqref{property:sigma}.

\begin{proposition}\label{prop:equivariance}
	Let $G$ be a finite group, with orthogonal representations $\rho\colon G\to O(d)$, $\tau\colon G\to O(2)$ and $\sigma\colon G\to\Sigma_n$ satisfying \eqref{property:sigma}. Then the following assertions hold true:
	\begin{itemize}
		\item[$(i)$] $\Lambda^G$ is a closed linear subspace of $\Lambda$;
		\item[$(ii)$] the Lagrange action functional $\mathcal{A}$ is $G$-equivariant, i.e.,
		\[
		\mathcal{A}(gx)=\mathcal{A}(x),\ \forall\,g\in G,\ \forall\,x\in\Lambda;
		\]
		\item[$(iii)$] the colliding set $\Delta$ is $G$-equivariant, i.e.,
		\[
		x\in\Delta\implies gx\in\Delta,\ \forall\,g\in G.
		\]
	\end{itemize}
\end{proposition}
As a consequence of these facts, we deduce this version of the \emph{Palais principle of symmetric criticality}, whose proof is a slight modification of the classical one given in \cite{Pal}.

\begin{lemma}\label{lem:palais}
	Let $\mathcal{A}|_{\Lambda^G}$ be the restriction of the Lagrange action functional $\mathcal{A}$ to the $G$-equivariant loop space $\Lambda^G$. Then, a collision-less critical point $\bar{x}\in\hat\Lambda^G$ of $\mathcal{A}|_{\Lambda^G}$ is also a (collision-less) critical point of $\mathcal{A}$ over the whole $\Lambda$. 
\end{lemma}
Now, we introduce the set of equivariant configurations
\[
	\mathcal{X}^G\uguale\left\lbrace x\in\mathcal{X}:\,\rho(g)x_{\sigma(g^{-1})i}=x_i,\ \forall\,i=1,\ldots,n\right\rbrace
\]
and recall this crucial result contained in \cite[Proposition 4.1]{FT2003}
\begin{proposition}\label{prop:coercive}
	Assume that $\mathcal{A}|_{\Lambda^G}$ is not identically $+\infty$. Then
	\[
	\mathcal{A}|_{\Lambda^G} \text{ is coercive } \iff \mathcal X ^G = \{0\}.
	\]
	As a consequence, if $\mathcal{X}^G=\{0\}$, there exists at least a minimiser of $\mathcal{A}$ in $\Lambda^G$.
\end{proposition}

\subsection{Reducible, bound to collisions and admissible actions on $\Lambda$}

It is convenient to refine the choice of a finite group $G$ with respect to its action on the periodic loops in $\Lambda$, in order to rule out those representations of $G$ which lead to \emph{reducible} group actions. This process, described in the present section, will yield a definition of \emph{admissible} group action on $\Lambda$. All the results and definitions of this paragraph have been introduced in \cite{FT2003}, or follows as direct consequences of such results. We will add some proofs for the reader's convenience. 

\begin{definition}\label{def:non_reducible}
	Consider a finite group $G$ with orthogonal representations $\rho\colon G\to O(d)$, $\tau\colon G\to O(2)$ and $\sigma\colon G\to\Sigma_n$ satisfying \eqref{property:sigma}. We say that the action of $G$ on $\Lambda$ is \emph{non-reducible} if the following assumptions are satisfied:
	\begin{itemize}
		\item[$(i)$] $\ker\tau\cap\ker\rho\cap\ker\sigma=1$;
		\item[$(ii)$] there is no proper linear subspace $V$ of $\R^d$ such that
		\[
		x_i(t)\in V,\quad\forall\,i\in\{1,\ldots,n\},\ \forall\,x\in\Lambda^G,\ \forall\,t\in\mathbb{T};
		\]
		\item[$(iii)$] there is no integer $k\neq\pm1$ such that
		\[
		\forall\,x\in\Lambda^G\ \exists\,y\in\Lambda\ \text{such that}\ x(t)=y(kt),\ \forall\,t\in\mathbb{T}.
		\]
	\end{itemize}
\end{definition}

\begin{remark}
	We can motivate the choices in the previous definition in this way:
	\begin{itemize}
		\item[$(i)$] If $K=\ker\tau\cap\ker\rho\cap\ker\sigma\neq 1$ then it is enough to consider $G'=G/K$ instead of $G$, since in this case $\Lambda^{G'}=\Lambda^G$ up to homeomorphisms.
		\item[$(ii)$] If $V$ is a proper subspace of $\R^d$ then $\dim V< d$ and one could change the representation $\rho$ of $G$.
		\item[$(iii)$] This assumption simply means that $T$ is the minimal period for equivariant loops. 
	\end{itemize}
\end{remark}

We can actually determine some situations in which the action of a group $G$ is reducible. 

\begin{proposition}\label{prop:reducible}
	Let $G$ be a finite group acting on $\Lambda$. If one of the following holds
	\begin{itemize}
		\item[$(a)$] $\ker\tau\cap\ker\sigma\neq 1$
		\item[$(b)$] $|\ker\rho\cap\ker\sigma|>2$
	\end{itemize}
	then the action of $G$ is reducible. 
\end{proposition} 
\begin{proof}
	Assume ($a$). Let $g\in(\ker\tau\cap\ker\sigma)\setminus\{1\}$. If $g\in\ker\rho$ then the action is reducible since $(i)$ in Definition \ref{def:non_reducible} fails. Otherwise, for $x\in\Lambda^G$ we have
	\[
	\rho(g)x_i(t)=x_i(t),\ \forall\,t\in\mathbb{T},\ \forall\,i\in\{1,\ldots,n\}
	\]
	and so $x_i(t)$ belongs to the proper subspace $V^g=\{v\in\R^d:\,\rho(g)v=v\}=\ker\rho(g)$; this is against $(ii)$ in Definition \ref{def:non_reducible}.
	
	Assume ($b$). Let $g_0,g_1\in(\ker\rho\cap\ker\sigma)\setminus\ker\tau$. Since $\tau(G)$ is a finite subgroup of $O(2)$, then $g_0$ and $g_1$ act as rotations or refections on $\mathbb{T}$ (see Remark \ref{rem:tau_O(2)}). If both act as reflections, their composition $\hat{g}\in \ker\rho\cap\ker\sigma$ and acts as a rotation. Then, if we consider $x\in\Lambda^G$, the equivariance with respect to $\hat{g}$ reads
	\[
	x_i(t)=x_i(t+\hat\vt),\quad\forall\,t\in\mathbb{T},\ i=1,\ldots,n,
	\]
	for some $\hat{\vt}\in\R/{T\Z}$. Since $x$ is a $T$-periodic loop, the previous relation implies that $T=k\hat{\vt}$ for some $k\in\Z\setminus\{0,\pm1\}$. In this case, $T$ would not be the minimal period for $x\in\Lambda^G$, and this would lead again to a reducible action of $G$ for condition $(iii)$ of Definition \ref{def:non_reducible}. Clearly 
\end{proof}

\begin{remark}\label{rem:brake}
	Note that if $|\ker\rho\cap\ker\sigma|=2$ the action of $G$ is not necessarily reducible. Indeed, the element $g\in(\ker\rho\cap\ker\sigma)\setminus\{1\}$ may act as a reflection, which would not imply property $(iii)$ of Definition \ref{def:non_reducible}.
\end{remark}

Since we are looking for classical solutions of the $n$-body problem, it is important to know that the action of some groups $G$ can make $\hat\Lambda^G$ an empty set, and \emph{forcing} the bodies to collide. 

\begin{definition}
	We say that the action of $G$ on $\Lambda$ is \emph{bound to collisions} if 
	\[
	\forall\,x\in\Lambda^G\ \text{it holds}\ x^{-1}\Delta\neq\emptyset,
	\]
	or, equivalently, if
	\[
	\hat{\Lambda}^G=\emptyset.
	\]
\end{definition}

The next easy result provides a necessary condition for $\hat{\Lambda}^G$ to be non empty. 

\begin{proposition}\label{prop:btc}
	Let $G$ be a finite group with non-reducible action on $\Lambda$. Then, if $\ker\tau\cap\ker\rho\neq 1$, the action of $G$ is bound to collisions. 
\end{proposition}
\begin{proof}    
    If $g\in(\ker\tau\cap\ker\rho)\setminus\{1\}$ then $\sigma(g)$ has to be non-trivial for condition $(i)$ of Definition \ref{def:non_reducible} and thus, if $x\in\Lambda^G$, we have that
	\[
	x_{\sigma(g^{-1})i}(t)=x_i(t),\ \forall\,t\in\mathbb{T},\ i=1,\ldots,n,
	\]
	which proves that any $G$-equivariant loop is a collision loop.
\end{proof}

	To conclude, in order to exclude those finite groups which \emph{surely} induce a reducible or bound to collisions action, we may further assume the following hypotheses on $G$, motivated by Propositions \ref{prop:reducible}-\ref{prop:btc} and give a definition of \emph{admissible} group action on $\Lambda$. We notice that an admissible action is not necessarily non reducible, nor not bound to collisions. Nonetheless, the properties required in the admissibility are easy to check and refine the choice of the groups $G$ we are interested in.

\begin{definition}[Admissible action]\label{def:admissible_action}
	We say that the action of a finite group $G$ on $\Lambda$ is \emph{admissible} if:
        \begin{itemize}
	    \item $\ker\tau\cap\ker\sigma=1$. This shows in particular that $G$ is a subgroup of $O(2)\times\Sigma_n$ through canonical isomorphism considering the product homomorphism $(\tau,\sigma)$;
     	\item $|\ker\rho\cap\ker\sigma|\leq 2$ and the non-trivial element in $\ker\rho\cap\ker\sigma$ acts as a reflection in time;
    	\item $\ker\tau\cap\ker\rho=1$. Again this shows that $G$ is a subgroup of $O(2)\times O(d)$.
\end{itemize}
\end{definition}

\subsection{Classification of the group action on $\Lambda$ and fundamental domain}

From now on, we assume that $G$ generates an admissible action on $\Lambda$ as described in Definition \ref{def:admissible_action}. At this point we aim to give a classification of finite groups which is based on the \emph{type of action} they provide on $\Lambda$. Actually, since $\tau(G)$ is a finite subgroup of $O(2)$, Proposition \ref{prop:subgroups_orthogonal} and Remark \ref{rem:tau_O(2)} suggest that the easiest and most significant classification has to take into account the action that $G$ yields on $\mathbb{T}$. In particular, $G/\ker\tau$ is itself a finite subgroup of $O(2)$ and it is useful to see how the action of $G$ on $\mathcal{X}$ is related to the one of $G/\ker\tau$ on $\mathcal{X}^{\ker\tau}$. To start with, let us define 
\[
	\bar{G}\uguale G/{\ker\tau}. 
\]
We recall that $G$ acts on the configuration space $\mathcal{X}$ through the representations $\rho\colon G\to O(d)$ and $\sigma\colon G\to\Sigma_n$. Since $\ker\tau\lhd G$, by Remark \ref{rem:weyl} we see that $\bar{G}$ acts on the fixed space $\mathcal{X}^{\ker\tau}$. The quotient $\bar{G}$ is a useful refinement of the whole group $G$, mostly because it is a subgroup of $O(2)$. This makes more convenient to work with $\bar{G}$-\emph{equivariant} loops which live in $\mathcal{X}^{\ker\tau}$ and this turns out to be also reasonable thanks to the next result. 
\begin{proposition}\label{prop:homom_kertau}
	The $H^1$ loop spaces 
	\[
	\Lambda^G=H^1(\mathbb{T};\mathcal{X})^G,\quad \Omega^{\bar{G}}\uguale H^1(\mathbb{T};\mathcal{X}^{\ker\tau})^{\bar{G}}
	\]		
	coincide pointwise.
\end{proposition}
\begin{proof}
	As a first remark, if $x\in\Lambda^G$ then $x(t)\in\mathcal{X}^{\ker\tau}$ for any $t\in\mathbb{T}$. Indeed, the $G$-equivariance of an element $x\in\Lambda^G$ implies in particular that
	\[
	hx_{h^{-1}i}(h^{-1}t)=x_i(t),\qquad\forall\,i\in\{1,\ldots,N\},\ \forall\,t\in\mathbb{T},
	\]
	and for any $h\in\ker\tau$, but obviously $h^{-1}t=t$. Moreover, the $G$-equivariance of $x$ clearly implies the $\bar{G}$-equivariance, so $\Lambda^G\sset\Omega^{\bar{G}}$.
	
    Assume now that $x\in\Omega^{\bar{G}}$; then, for any $g\in G$, $h\in\ker\tau$ we have that
    \begin{equation}\label{eq:equiv_ker_tau}
    \rho(gh)x_{\sigma((gh)^{-1})}(\tau((gh)^{-1})t)=x_i(t)
    \end{equation}
    or
    \[
    \rho(g)\rho(h)x_{\sigma(h^{-1})\sigma(g^{-1})}(\tau(h^{-1})\tau(g^{-1})t)=x_i(t).
    \]
    Now, if we call $s=\tau(g^{-1})t$ and $j=\sigma(g^{-1})i$ we can rewrite the left hand side as
    \[
    \rho(g)\rho(h)x_{\sigma(h^{-1})j}(\tau(h^{-1})s).
    \]
    Now, recall that $x(t)\in\mathcal{X}^{\ker\tau}$ for any $t\in\mathbb{T}$ and that $h\in\ker\tau$, so that $\tau(h^{-1})s=s$. Therefore, we can rewrite the previous expression as
    \[
    \rho(g)x_j(s)
    \]
    and now, since $j=\sigma(g^{-1})i$ and $s=\tau(g^{-1})t$ we have actually proved that \eqref{eq:equiv_ker_tau} holds if and only if
    \[
    \rho(g)x_{\sigma(g^{-1})i}(\tau(g^{-1})t)=x_i(t),\quad\text{for any}\ g\in G,
    \]
    and so $x\in\Lambda^G$.
\end{proof}
Note that, since $\bar{G}$ is a finite subgroup of $O(2)$, then Proposition \ref{prop:subgroups_orthogonal} guarantees that $\bar{G}$ is either \emph{cyclic} or \emph{dihedral}. This motivates the following definition.

\begin{definition}
	We classify the action of a finite group $G$ on $\Lambda$ in this way:
	\begin{itemize}
		\item if the quotient $\bar{G}$ acts trivially on the orientation of $\mathbb{T}$, then $\bar{G}$ is cyclic and we say that the action of $G$ on $\Lambda$ is of \emph{cyclic type}. In this case, $\bar{G}$ is a subgroup of $SO(2)$ and only contains rotations;
		\item if $\bar{G}$ is made by a single reflection on $\mathbb{T}$, then we say that the action of $G$ on $\Lambda$ is of \emph{brake type}. In particular, $\bar{G}$ is a cyclic group of order 2;
		\item if none of the previous is satisfied we say that the action of $G$ on $\Lambda$ is of \emph{dihedral type} and $\bar{G}$ is a dihedral subgroup of $O(2)$.
	\end{itemize}
\end{definition}
The peculiarity of symmetric loops is that, once the group $G$ is fixed, it is possible to detect a sub-interval of $\mathbb{T}$ which characterises the entire loop. This is a fundamental step when one looks for critical points of $\mathcal{A}$, since it allows to switch from the analysis of periodic loops to the one of fixed-ends paths which share initial and final conditions fixed on the extrema of such subinterval. For this reason, we firstly give the following definition. 
\begin{definition}
	The isotropy subgroups of the action of $G$ on $\mathbb{T}$ are the sets
	\[
		G_t=\{g\in G:\,\tau(g)t=t\},\quad \text{for}\ t\in\mathbb{T},
	\]
	and are called the $\mathbb{T}$\emph{-isotropy subgroups} of $G$.
\end{definition}

From the previous definition, the $\mathbb{T}$-isotropy subgroups of $G$ correspond to the isotropy subgroups of the action of $\bar{G}$ on $\mathbb{T}$. This information is very useful, since we can classify again the $\mathbb{T}$-isotropy subgroups of $G$ as subgroups of $O(2)$.
\begin{remark}\label{rem:isotropy_subgroups}
	The number of distinct $\mathbb{T}$-isotropy subgroups completely determines the action type of $G$. At first, notice that $\ker\tau$ is always a $\mathbb{T}$-isotropy subgroup of $G$. Moving to the action of $\bar{G}$, observe that any $\bar{G}_t$ cannot contain rotations. Denoting by $l$ the number of distinct $\mathbb{T}$-isotropy subgroups, we obtain this characterisation:
	\begin{itemize}
		\item $l=1$, i.e., the maximal $\mathbb{T}$-isotropy subgroup is $\ker\tau$ $\iff$ the action is \emph{cyclic};
		\item $l=2$, i.e., the $\mathbb{T}$-isotropy subgroups are $\ker\tau$ and a maximal one $\iff$ the action is of \emph{brake type}. In this case, the maximal $\mathbb{T}$-isotropy subgroup coincides with $G$;
		\item $l\geq 3$, i.e., the $\mathbb{T}$-isotropy subgroups are either maximal or $\ker\tau$ $\iff$ the action is of \emph{dihedral type}. In this case, every maximal $\mathbb{T}$-isotropy subgroup is cyclic of order 2. This can be seen by working with the isotropy subgroups of the action of $\bar{G}$ and observing that no rotations can belong to any $\bar{G}_t$. Note that the product of two reflections is a rotation. 
	\end{itemize}
\end{remark}

We now define the already mentioned sub-interval of $\mathbb{T}$.

\begin{definition}[Fundamental domain]\label{def:fundamental_domain}
	The \emph{fundamental domain} $\mathbb{I}\subset\mathbb{T}$ of the action of $G$ on $\mathbb{T}$ is defined as follows:
	\begin{itemize}
		\item if the action of $G$ is cyclic, then $\mathbb{I}$ is the closed interval which connects the instant $t=0$ with its image $\tau(g^{-1})0$, where $g$ is the cyclic generator of $\bar{G}$;
		\item  if the action of $G$ is brake or dihedral, then $\mathbb{I}$ is a closed interval whose extrema are two distinct elements of $\mathbb{T}$ generating two distinct maximal $\mathbb{T}$-isotropy subgroups, such that no other instants related to maximal $\mathbb{T}$-isotropy subgroups are included in $\mathbb{I}$. Note that, if the action is of brake type, the unique maximal $\mathbb{T}$-isotropy subgroup is $G$ itself. 
	\end{itemize}
    
    \noindent In particular, there are $\vert\bar{G}\vert$ of such domains and we have
    \[
    \mathbb{T}=\bigcup_{[g]\in\bar G}\tau(g^{-1})\mathbb{I}\quad\text{and}\quad \vert\mathbb{I}\vert=\frac{\vert\mathbb{T}\vert}{\vert\bar{G}\vert}.
    \]
\end{definition} 

The previous definition may seem rather technical and difficult to picture, but the following example should clarify its important meaning. 

\begin{example}[Fundamental domain of $D_6$]
	\label{ex:d6_2}
	Recalling that $D_6$ is generated by a rotation $r$ and a reflection $s$ (see Example \ref{ex:d6_1}), we kwnow that
	\[
	D_6=\left\lbrace 1,s,r,r^2,sr,sr^2\right\rbrace.
	\]
	In order to determine a fundamental domain $\mathbb{I}$, we need to detect first the the $\mathbb{T}$-isotropy subgroups of $D_6$. We observe that the rotations $r,r^2$ do not fix any element of $\mathbb{T}$. Direct computations show that within a period of $2\pi$:
    \begin{itemize}
    \item $s$ fixes the instants $t=\frac{\pi}{4},\frac{5\pi}{4}$;
    \item $sr$ fixes the instants $t=\frac{11\pi}{12},\frac{23\pi}{12}$;
    \item $sr^2$ fixes the instants $t=\frac{7\pi}{12},\frac{19\pi}{12}$. 
    \end{itemize} 
    We conclude that $D_6$ has three distinct non-trivial $\mathbb{T}$-isotropy subgroups
    \begin{itemize}
    	\item $G_{\frac{\pi}{4}}=G_{\frac{5\pi}{4}}=\langle s\rangle $
    	\item $G_{\frac{11\pi}{12}}=G_{\frac{23\pi}{12}}=\langle sr\rangle $
    	\item $G_{\frac{7\pi}{12}}=G_{\frac{19\pi}{12}}=\langle sr^2\rangle $
    \end{itemize}
    which are cyclic of order 2. Following Definition \ref{def:fundamental_domain}, we see that for instance
    \[
    \mathbb{I}=\left[\dfrac{\pi}{4},\dfrac{7\pi}{12}\right],\quad |\mathbb{I}|=\frac{2\pi}{6}.
    \]
\end{example}

\section{ $\boldsymbol G$-equivariant critical points of the action functional}\label{sec:optimisation}

This final section explains the numerical optimisation process which has been prepared until now. We recall that we want to optimise the action functional
\[
	\mathcal{A}(x)=\frac12\int_0^T \sum\limits_{i=1}^nm_i\|\dot{x}_i(t)\|^2\,dt +\int_0^T\sum\limits_{i<j}\frac{m_im_j}{\|x_i(t)-x_j(t)\|}\,dt
\]
over the symmetric loop space $\Lambda^G$. In particular, we firstly present the theoretical approach which permits to switch from the optimisation of $\mathcal{A}$ on $\Lambda^G$ to a suitable \emph{fixed-ends} problem depending on the symmetry. Secondly, we will deepen the numerical procedures which produce $G$-equivariant periodic orbits for the symmetrical $n$-body problem. We start with this first result.

\begin{proposition}\label{prop:homo_fund_domain}
	Consider a finite group $G$ and the usual representations $\rho,\tau$ and $\sigma$. Let $\mathbb{I}=[t_0,t_1]$ be the fundamental domain of the action of $G$.
	\begin{itemize}
		\item[$(i)$] If the action of  $G$ is brake or dihedral, let $H_0$ and $H_1$ be the two maximal $\mathbb{T}$-isotropy subgroups associated respectively with $t_0$ and $t_1$. Then, the space $\Omega^{\bar{G}}=H^1(\mathbb{T};\mathcal{X}^{\ker\tau})^{\bar{G}}$ is homeomorphic to the space of fixed-end paths
		\[
			Y=\{x\in H^1(\mathbb{I};\mathcal{X}^{\ker\tau}):\,x(t_0)\in\mathcal{X}^{H_0},\ x(t_1)\in\mathcal{X}^{H_1}\}.
		\]
		\item[$(ii)$] If the action of $G$ is cyclic, then $\Omega^{\bar{G}}$ is homeomorphic to the space
		\[
			Y=\{x\in H^1(\mathbb{I}; \mathcal{X}^{\ker\tau}):\, gx(t_0)=x(t_1)\},
		\] 
		where $g$ is the generator of $\bar{G}$.
	\end{itemize}
\end{proposition}
\begin{proof}
	We start by defining the map $\pi\colon\Omega^{\bar{G}}\to Y$ which associates to any loop $x\in\Omega^{\bar{G}}$ its restriction to the fundamental domain $\pi(x)\uguale x|_{\mathbb{I}}$. The map $\pi$ is clearly well-defined: the boundary conditions are satisfied in both cases, since $x$ is $\bar{G}$-equivariant. Note that the boundary conditions are different since when $\bar{G}$ is cyclic the unique maximal $\mathbb{T}$-isotropy subgroup is $\ker\tau$
	
	On the other hand, for any $y\in Y$ we can construct a unique path in $\Omega^{\bar{G}}$ by concatenating the symmetrised paths $gy$, for any $g\in\bar{G}$. Depending on the action type, anyone of these paths is defined on a reflection or rotation $g\mathbb{I}$ of the fundamental domain. Recalling that
	\[
	\mathbb{T}=\bigcup_{[g]\in\bar G}g\mathbb{I}
	\]
	the proof is complete.	
\end{proof}

At this point, for any action type, we can define the restriction of the action functional to the fundamental domain $\mathbb{I}$ on $Y$ in this way 
\[
\begin{aligned}
\mathcal{A}_{\mathbb{I}}&\colon Y\longrightarrow\R\cup\{+\infty\} \\
&y\mapsto\mathcal{A}_{\mathbb{I}}(y)\uguale\int_{\mathbb{I}}L(y(t),\dot{y}(t))\,dt.
\end{aligned}
\]
Arguing as in the proof of Proposition \ref{prop:homo_fund_domain}, if $y\in Y$ and $y=\pi(x)$, where $x$ is the concatenation of symmetrised paths $gy$, using also the equivariance of the action functional proved in Proposition \ref{prop:equivariance}, it is easy to check that
\begin{equation}\label{eq:proportional}
\mathcal{A}(x)=|\bar{G}|\mathcal{A}_{\mathbb{I}}(y).
\end{equation}
The following result establishes a relation between collision-less critical points of $\mathcal{A}_{\mathbb{I}}$ and symmetric solutions of the $n$-body problem.
\begin{theorem}
	Let $G$ be a finite group acting on $\Lambda$ and let $\mathbb{I}$ be the foundamental domain of its action. Assume that there exists a collision-less critical point $y\in Y$ of the restricted action functional $\mathcal{A}_{\mathbb{I}}$. Then, the symmetrised path $\pi^{-1}y$ is a solution of the $G$-equivariant $n$-body problem. 
\end{theorem}
\begin{proof}
	Consider a critical point $y\in Y$ and the symmetrised path $x=\pi^{-1}y$, which is the concatenation of the fixed-end paths $gy$, for any $g\in G/{\ker\tau}$. This is clearly a $G$-equivariant path, which is also collision-less by assumption. Since $y$ is a critical point of the restricted action functional on $Y$, by straightforward computations and Proposition \ref{prop:homo_fund_domain} we see that the symmetrised path $x$ is a critical point for $\mathcal{A}$ on $\Lambda^G$, and actually a collision-less free critical point on the whole $\Lambda$ (see Lemma \ref{lem:palais}). We are only left to prove that $x$ is $\mathcal{C}^1$ on the whole $\mathbb{T}$, since singularities may appear on the boundary of $g\mathbb{I}$ when we symmetrise $y$. We split the proof in two cases: 
	\begin{itemize}
	\item brake or dihedral action;
	\item cyclic action.
    \end{itemize}
	When the action is brake or dihedral, we exploit the criticality of $y$ in $Y$, which imposes some boundary conditions: not only that $y(t_0)\in\mathcal{X}^{H_0}$ and $y(t_1)\in\mathcal{X}^{H_1}$, but also that $M\dot{y}(t_0)\perp\mathcal{X}^{H_0}$ and $M\dot{y}(t_1)\perp\mathcal{X}^{H_1}$, where $M\in\R^{nd}\times\R^{nd}$ denotes the diagonal matrix of the masses of the bodies. To see that, it is enough to expand the equation $d\mathcal{A}_{\mathbb{I}}(gy)[\vp]=0$, for any $\vp\in Y$. 
	
	Now assume that $G/{\ker\tau}=\langle g:\,g^k=1\rangle$, for some $k\in\N$. Following Definition \ref{def:fundamental_domain}, in this case we have $\mathbb{I}=[0,T/k]$ and $H_0=H_1=\ker\tau$. This means that the path $y(t)\in \mathcal{X}^{\ker\tau}$ for any $t\in[0,T/k]$ and the only condition to be satisfied in order to have a $\mathcal{C}^1$ junction is that
	\[
	\dot{y}(T/k)=g\dot{y}(0).
	\]
	Again, this follows from the criticality condition on $\mathcal{A}_{\mathbb{I}}$.
\end{proof}

As already remarked in the introduction, the search for critical points of the action functional $\mathcal{A}$ is also motivated by a simple condition which guarantees the coercivity of $\mathcal{A}$ (see Proposition \ref{prop:coercive}). There exists also a \emph{local} version of this result, which establishes the coercivity of $\mathcal{A}_{\mathbb{I}}$ on fixed-ends paths under the same condition on $G$.
\begin{lemma}\label{lem:local_coercivity}
	The functional $\mathcal{A}_{\mathbb{I}}$ is coercive on $Y$ if and only if $\mathcal{X}^G=\{0\}$.
\end{lemma}
\begin{proof}
	By Proposition \ref{prop:coercive}-\ref{prop:homom_kertau}, we prove that $\mathcal{A}_\mathbb{I}$ is coercive if and only if $\mathcal{A}|_{\Omega^{\bar{G}}}$ is. Moreover, Proposition \ref{prop:homo_fund_domain} guarantees that, for any $x\in\Omega^{\bar{G}}$ there exists a constant $C>0$, depending exclusively on $|\bar{G}|$, such that
	\[
	\|\pi(x)\|_{H^1}=C\|x\|_{H^1}.
	\]
	This clearly provides the assertion.
\end{proof}
This quite clean condition motivates the search for \emph{local minimisers} of $\mathcal{A}_{\mathbb{I}}$, and it is straightforward to show that the periodic symmetrisation of these paths minimises $\mathcal{A}$ on $\Lambda$. In addition, any candidate minimiser of $\mathcal{A}_{\mathbb{I}}$ lives in the space $Y$ introduced in Proposition \ref{prop:homo_fund_domain} numerical procedures since now we are woking with boundary conditions. Nevertheless, a suitable definition of \emph{local minimiser} has to be given and major assumptions on $G$ are required in order to prove the absence of collisions. We prefer to postpone a survey on such results to Appendix \ref{app:rcp}.

\subsection{Approximating critical points of $\mathcal{A}$}
At this point, we present the numerical procedure which succeeds in finding some examples of such solutions as critical points of the action functional. The first version of this optimisation routine was introduced in \cite{symorb} and now it is also implemented in the Julia package \texttt{SymOrb.jl}. 

\subsubsection{Configuration space and projectors}

As a first step, we need to define a finite-dimensional approximation of the space $Y$ introduced in Proposition \ref{prop:homo_fund_domain}. The space $Y$ already contains two finite-dimensional spaces, the starting and final manifolds $\mathcal{X}^{H_0}$ and $\mathcal{X}^{H_1}$ (which both coincide with $\mathcal{X}^{\ker\tau}$ when the action is cyclic), and thus it is trivially diffeomorphic to the space 
\[
H^1(\mathbb{I};\mathcal{X}^{\ker\tau})\times\mathcal{X}^{H_0}\times\mathcal{X}^{H_1}.
\]
A natural way to approximate elements in $H^1(\mathbb{I};\mathcal{X}^{\ker\tau})$ is by using truncated Fourier series, while the manifolds $\mathcal{X}^{H_i}$ are already embedded in $(\R^d)^n$. As a result, our optimization procedure will take place in the linear space
\[
((\R^d)^n)^F\times (\R^d)^n\times (\R^d)^n ,
\]
where $F\in\N$ stands for the length of Fourier polynomials. 
From this point, we assume that $\mathbb{I}=[0,\pi]$. Consider $A_1, \ldots A_F \in \qty(\R^d)^n$ as the Fourier coefficients and $A_0=x_0, A_{F+1}=x_1\in(\R^d)^n$ as the initial and final configurations. In this way, an approximating path will be defined as
\begin{equation}
	y(t)\uguale x(t)+s(t),
	\label{eq:path}
\end{equation} where $x(t)$ stands for the segment between $x_0$ and $x_1$
\[
x(t)=x_0+\frac{t}{\pi}(x_1-x_0),\quad\text{for}\ t\in[0,\pi],
\]
and $s(t)$ is the truncated Fourier series
\[
s(t)\uguale \sum\limits_{k=1}^FA_k\sin(kt),\quad\text{for}\ t\in[0,\pi].
\]
Consider a discretisation $(t_h)_{h=0\ldots S}$ of the interval $[0,\pi]$, where $t_0=0$ and $t_S=\pi$. Then $\Delta t = \frac \pi S$ and $t_h = h \Delta t$, so that
\[
x(t_h) = x_0 + \frac{h}{S} (x_1 - x_0)\] and \[s(t_h) = \sum_{k=1}^F A_k \sin(\frac \pi S k h).
\]

The functions we want to approximate pointwisely belong to the space $\mathcal{X}^{\ker\tau}$ and also the boundary conditions have to be imposed (note again that they change according with the action type). Therefore, recalling the definition of projector on fixed spaces given in \eqref{eq:projector}, we consider the maps
\[
	\begin{aligned}
		&\pi_{\ker\tau}\colon (\R^d)^n\to \mathcal{X}^{\ker\tau} \\
		&\pi_{H_i}\colon (\R^d)^n\to \mathcal{X}^{H_i},\quad i=0,1,
	\end{aligned}
\]
which project any configuration in the correct equivariant space. Looking at the definition of $Y$ in Proposition \ref{prop:homo_fund_domain}, the cyclic case deserves a different treatment. In that case, the boundary conditions requires an equivariance property, which can be translated by using the following projector
\[
\begin{aligned}
	\pi_g\colon(\R^d)^n\times&(\R^d)^n\longrightarrow V_g \\
	& (v,w)\longmapsto \pi_g(v,w)\uguale \left(\frac12(v+g^{-1}w), \frac12(gv+w)\right)
\end{aligned}
\]
where $g$ is the cyclic generator of $\bar{G}$ and $V_g\uguale\{(v,w)\in(\R^d)^n\times(\R^d)^n:\,gv=w\}$.
With these definitions, the projector of a pair of points on the right boundary condition space is the following
\[
   \pi_{bc}(v,w)=\begin{cases}
				\begin{aligned}
				&(\pi_{H_0}v,\pi_{H_1}w) &\text{if the action is brake/dihedral} \\
				&\pi_g(v,w)  &\text{if the action is cyclic} 
				\end{aligned}
   				\end{cases}
\]
\begin{remark}[Fixing the centre of mass]\label{rem:com}
	We recall that the centre of mass in $\mathcal{X}$ is fixed at the origin. In \texttt{SymOrb} (\cite{symorb}), to accomplish this requirement, at every step of the optimization process the whole configuration is suitably translated. In \texttt{SymOrb.jl} we propose a different approach: only $n-1$ configurations are actually free variables, while the $n$-th is a function of the former so that the center of mass remains at the origin. Beside avoiding a projection at each step, this method considerably reduces the number of parameters to be optimized, hence shortening the computational time, especially for low number of bodies.  
\end{remark}

\subsubsection{The action functional}\label{subsec:action_numerical}

Each optimization step requires the evaluation of the action functional $\mathcal{A}$ introduced in \eqref{eq:action} on a given path, together with its gradient (if the optimisation method is a first order one) and its Hessian (if a second order method is used). It is possible to compute them without approximating the velocities $\dot y(t)$ and the accelerations $\ddot y(t)$ of the path $y(t)$ in a (possibily) rotating reference frame with infinitesimal rotation matrix $\Omega \in \mathfrak{so}(d)$, the Lie algebra of $SO(d)$.

We introduce the kinetic and the potential components of the action functional on $[0,\pi]$
\[
	\mathcal A(y) = \mathcal K(y, \dot y) + \mathcal U(y) =  \int_0^\pi K(y, \dot y)\,dt  + \int_0^\pi U(y)\,dt;
\]
note that the kinetic part here depends not only on $\dot{y}$, but also on $y$, because we can also consider the $n$-body problem in a rotating frame (see also Remark \ref{rem:general_problem}).
In the following, we are going to compute the gradient and the Hessian of $\mathcal{A}$ and, since we are working with approximed paths, such operators will be computed by differentiating with respect to the coefficients matrices $A_0,A_1,\ldots,A_F, A_{F+1}\in(\R^d)^n$. For this reason, we will often write $y(A)$ and $\mathcal{A}(A)$ to point out the dependence on such coefficients matrices. 

\paragraph{The kinetic part}
Consider, in a possibly rotating reference frame, a path $y(t)$ as in \eqref{eq:path}  and its velocity $\dot y(t)$. Let $\tilde{y}(t)$ be the same path in the inertial frame, so that $\dot {\tilde{y}}(t) = \dot y(t) - \Omega y$. The kinetic energy is then

\begin{align*}
	K(y, \dot y) &= \frac 1 2 \sum_{i=1}^n m_i \| \dot{\tilde{y}}_i(t)  \|^2 = \frac 1 2 \sum_{i=1}^n m_i \| \dot y_i(t) - \Omega y(t) \|^2 \\ 
	&= \frac 1 2 \sum_ {i = 1}^n m_i \Big[\underbrace{\dot y_i^t \dot y_i}_\text{linear}  - \underbrace{y_i^t \Omega^2 y_i}_\text{centrifugal} + \underbrace{\left(\dot y_i^t \Omega y_i - y_i^t \Omega \dot y_i\right)}_\text{Coriolis}  \Big]\\ 
	&= K_\text{lin}(\dot y) + K_\text{centr}(y) + K_\text{Cor}(y, \dot y).
\end{align*} 

Integrating from $0$ to $\pi$, since the dependence of $y$ on the coefficients $A_k$ is well understood, the kinetic part of the action can be written as a quadratic form $\mathcal K : ((\R^n)^d)^{F+2}  \to \R $,
such that $\mathcal K(A) =\frac 1 2  y^t(A)\,\mathbf K \,y(A)$ and $\mathbf K = \mathbf K_\text{lin} + \mathbf K_\text{centr} + \mathbf K_\text{Cor}$. In particular, it is possible to rewrite the three components as 
\begin{align*}
	&\mathbf K_\text{lin} = (K^\text{lin}_{jk}) \otimes (\mathbf M \otimes \mathbf{Id}_d),  \\ 
	&\mathbf K_\text{centr} = (K^\text{centr}_{jk}) \otimes (\mathbf M \otimes \Omega^2), \\
	&\mathbf K_\text{Cor}  = (K^\text{Cor}_{jk}) \otimes  (\mathbf M \otimes \Omega),
\end{align*}
where $K^\text{lin}_{jk}, K^\text{centr}_{jk},  K^\text{Cor}_{jk}$ are $(F+2)\times(F+2)$ matrices, $ \mathbf M = \text{diag}(m_1 , \ldots, m_n)$ is the diagonal matrix of masses and $\otimes$ is the Kronecker product.

Hence, the gradient of the kinetic part of the action functional is given by $\nabla \mathcal K(A) = \mathbf K\, y(A)$ and the Hessian is $\nabla^2 \mathcal K(A) = \mathbf K$.

\paragraph{The potential part} Consider the potential function $U(y)$. It can be written as 
\[ 
	U(y) = \sum\limits_{i = 1}^n\sum\limits_{j = i+1}^n \frac{m_i m_j}{f\left(\| y_{ij}\|\right)}  =\sum\limits_{i = 1}^n\sum\limits_{j = i+1}^n  U_{ij}(y_{ij}) \] where $y_{ij} = y_i - y_j$, $f : \R^+ \to \R$ is a function of the norm $\|y_{ij} \|$. Formally, the components of the gradient and of the Hessian $\pdv{U_{ij}(y)}{y_i} \in \R^d$ and $\pdv{U_{ij}(y)}{y_i}{y_j} \in \R^{d \times d}$ are given by 
\[
	\pdv{U_{ij}(y)}{y_i} = - m_i m_j\frac{f'}{f^2} \frac{y_{ij}}{\| y_{ij}\|}
\]
\[
	\pdv{U_{ij}(y)}{y_i}{y_j} =  \frac{m_i m_j}{f^2\cdot \|y_{ij}\|^2} \left[ ( y_{ij} \otimes y_{ij}) \left( f'' - \frac{f'}{\|y_{ij}\|}  -  \frac{2f'^2}{f}\right) + (f' \cdot \|y_{ij}\|) \text{Id}_d\right]
\]
where the dependence of $f$  on $y_{ij}$ is understood. Hence, for $i=1,\ldots,n$ we have
\begin{align*}
	\pdv{U(y)}{y_i} = \sum_{j \neq i} \pdv{U_{ij}(y)}{y_i} 
	\qquad\pdv[2]{U(y)}{y_i} = \sum_{j \neq i} \frac{\partial^2 U_{ij}(y)}{\partial y_i^2},
\end{align*}
while, for $i,j\in\{1,\ldots,n\}$, $i\neq j$ we obtain
\[
	\pdv{U(y)}{y_i}{y_j} = \frac{\partial}{\partial y_j}\left(\frac{\partial U(y)}{\partial y_i}\right)=\pdv{U_{ij}(y)}{y_i}{y_j}.
\]
Therefore, integrating in $[0,\pi]$, for $k,h=0,\ldots,F+1$ we finally have
 \begin{align*}
	&\mathcal U(A) = \int_0^\pi U(y)dt \\ 
	&\grad \mathcal U(A)_k = \int_0^\pi \sum_{i = 1 \ldots n} \pdv{U(y)}{y_i} \pdv{y_i}{A_k} dt \\
	&\hess \mathcal U(A)_{kh} = \int_0^\pi \sum\limits_{i=1}^n\left(\sum\limits_{j=1}^n\pdv[2]{U(y)}{y_i}{y_j}\pdv{y_j}{A_h}\right)^t \pdv{y_i}{A_k}  dt
\end{align*}
where the integration is performed numerically. Note that $\frac{\partial y_i}{\partial A_k}$ is independent of $A_k$ and can be precomputed.

\begin{remark}\label{rem:com2}
	As explained in Remark \ref{rem:com}, analyzing the code it turns out that each matrix involved in the computation actually refers to $n-1$ bodies and already encodes the dependence of the last one from the others. For the sake of clarity and consistency with the theoretical framework, we prefer to not change the dimensions neither in the next section, understanding that the $n$-th body is a function of the others $n-1$.
\end{remark}

\subsection{Numerical algorithm}

Below we collect the main steps involved in the numerical optimization routine:
\begin{itemize}
	\item Consider an initial sequence
	\[
	A_0, A_1,\ldots, A_F, A_{F+1}\in (\R^d)^n,
	\]
	where $A_0$ and $A_{F+1}$ correspond to the initial and final configurations, while $A_1,\ldots,A_F$ represent Fourier coefficients. This sequence can be random or chosen ad-hoc to improve the convergence of the algorithm.
	\item Project $A_0$ and $A_{F+1}$ onto the boundary manifolds $\mathcal{X}^{H_0}, \mathcal{X}^{H_1}$ using the map $\pi_{bc}$ and thus obtaining the initial and final configurations
	\[
		(x_0,x_1)=\pi_{bc}(A_0,A_{F+1}).
	\]
    Then project $x_0, x_1, A_1, \ldots,A_F$ onto $\mathcal X^{\ker\tau}$ using the map $\pi_{\ker\tau}$. Note that, since the path depends linearly on the Fourier coefficients, it is possible to project the latter instead of the former.
	\item Optimise the action functional, written in terms of Fourier coefficients on the fundamental domain $\mathbb I = [0,\pi]$, using a suitable numerical method that stops when a convergence conditions is fulfilled (see in particular Section \ref{sec:sym_guide}).  
	At each optimisation step, before computing the action, its gradient and its Hessian, project $x_0,x_1,A_1, \ldots A_F$ as before. Moreover, project also the gradient and the Hessian of the action, to ensure that the optimisation is performed in the correct space. The gradient is projected using the same projectors used for the Fourier coefficients. The Hessian is projected using the same projectors used for the Fourier coefficients applied to the rows and columns of the Hessian.
\end{itemize}

	\subsubsection{Extension to the whole period}
	The optimized path is defined on the fundamental domain $\mathbb{I}=[0,\pi]$ via Fourier coefficients and can be extended to the whole period $\mathbb{I} = [0, m\pi]$, where $m = |\bar G|$, using the symmetries of the group and an inverse Fourier transform.
	In fact, consider $y(t)$ as in \eqref{eq:path} where the coefficients $A_0, \ldots, A_{F+1}$ are the optimized ones. Let ${y_h  \uguale y(t_h)}, 0\leq h \leq S$ be its discretisation. 
	\paragraph{Cyclic action} In the cyclic case, the group $\bar G$ is generated by $g\in G$. The path can be extended up to $y_{mS}$ by

	\[ 
	y_{jS + k} = \rho(g^j)y_k \qquad \text{for } 1\leq j \leq m, 0\leq k < S.
	\]

	\paragraph{Dihedral or brake action} In the dihedral or brake case, the group $\bar G$ is generated by $g_0 \in H_0$ and $g_1 \in H_1$. The path can be extended up to $y_
	{2S}$ by
	\[
		y_{S+k} = \rho(g_1)y_{S-k} \qquad \text{for } 0 \leq k < S.
	\]
	Then it can be extended to the whole period, up to $mS$, analogously to the cyclic case, replacing $S$ with $2S$. Clearly, if the action is of brake type, this last step is not needed.

\section{The basics of \texttt{SymOrb.jl}}\label{sec:sym_guide}

This section contains a short guide through the main steps required to implement an optimisation routine in \texttt{SymOrb.jl}. An exhaustive and practical guide through the up-to-date version of the software can be reached at \url{link_to_github}, endowed with sessions examples and tips. 

\subsection{Wrapping the group information}

Once the dimension $d$ of the configuration space and the number of body $n\ge 3$ are fixed, the equivariant optimisation process is based on the choice of a symmetry group $G$. As widely described in Section \ref{sec:sym_lambda}, we recall that the action of $G$ is represented through three homomorphisms $\rho\colon G\to O(d), \tau\colon G\to O(2)$ and $\sigma\colon G\to \Sigma_n$, and that the quotient $G/\ker\tau$ is a finite subgroup of $O(2)$. In the previous sections we saw that the optimisation routine starts with the determination of the boundary manifolds $\mathcal{X}^{H_0}$ and $\mathcal{X}^{H_1}$ and the connection manifold $\mathcal{X}^{\ker\tau}$. From a numerical perspective, to determine such optimisation spaces where the projectors act, we need to specify:
\begin{itemize}
\item the action type (cyclic, dihedral or brake);
\item the representations of the two generators of $G/\ker\tau$, each of them as a pair composed by a matrix in $O(d)$ and a permutation in $\Sigma_n$. If the action is of dihedral or brake type, we need an actual pair of generators: one acting as a rotation, the second as a reflection. When the action is cyclic, the two generators are the same and act as a rotation;
\item the generators of $\ker\tau$, again as pairs of matrices and permutations. Note that, if $\ker\tau$ is trivial no generators have to be specified.
\end{itemize}
With this information, we can reconstruct the whole group $G$ using an ad-hoc \texttt{GAP} library, whose original version \texttt{SymOrb.g} has been created by Davide Ferrario and can be reach at \cite{symorb}. 

In addition, we need to pass as input the following parameters:

\begin{itemize}
	\item the masses of the bodies as positive real numbers;
	\item the $d\times d$ anti-symmetric matrix $\Omega\in \mathfrak{so}(d)$, which is the infinitesimal generator of the angular velocity as explained in Section \ref{subsec:action_numerical} (see also Remark \ref{rem:general_problem});
	\item the number of Fourier coefficients $F$;
	\item the denominator $f\colon\R^+\to\R$ of the potential function $U$ (see Remark \ref{rem:general_problem}). 
\end{itemize}
Let us remark that it is convenient, but not necessary, to choose $G$ in such a way that the admissibility condition on the action of $G$ is satisfied (see Definition \ref{def:admissible_action}).

\subsection{Optimisation routines}

The next step is to choose a starting path, from which the optimisation routine begins. This actually means to provide $F$ Fourier coefficients, and the initial and final configurations (as described in Section \ref{sec:optimisation}). The software allows to proceed with the natural choice of random coefficients, but also to exploit user-provided initial paths, such as a circular guess, in which the bodies are equally spaced. 

At this point, the optimisation routine can finally start; available methods implemented in the Julia package are:
	\begin{itemize}
	\item \textbf{First order methods}: Conjugate-Gradient, Gradient descent and BFGS, which are all implemented in the Julia package \texttt{Optim.jl} \cite{Optim.jl};
	\item \textbf{Second order methods}: second order methods are used  to find the solutions of $\grad \mathcal A = 0$. The impelemented method is a Newton method, both with line search and trust region, from the Julia package \texttt{NLSolve.jl} \cite{NLSolve.jl}.
	\end{itemize}
The choice of the method requires also to fix the maximum number of iterations and a condition under which the convergece is achievable. 
Note that such methods can also be combined, in order to improve the convergence to non-trivial Morse index critical points hidden in some regions of the competitors space.

\subsection{Output and visualisation}

Once a numerical critical point is found, the output consists of the list of Fourier coefficients on the fundamental domain. The actual trajectories of the bodies are then reconstructed exploiting the symmetry properties and using an inverse Fourier transform (see Section \ref{sec:optimisation}). If $d=2,3$, it is also possible to visualise the symmetric orbit as an animation, using a graphic engine implemented with the library \texttt{Makie.jl} (see \cite{DanischKrumbiegel2021}).

\subsection{Example session}
\label{sec:example_session}
Following the examples described above in \ref{ex:d6_1} and \ref{ex:d6_2}, we provide a simple session to show how to use \texttt{SymOrb.jl} to find a critical point of the action functional with a prescribed symmetry. In the following, it is assumed that Julia and the \texttt{SymOrb.jl} package, \cite{SymOrbJL}, are already installed. 

The first step is to create a \texttt{TOML} file which contains the information about the group, the masses, the number of Fourier coefficients and the potential function. For $3$ bodies in the plane under the action of a $D_6$ symmetry and Newtonian potential the file \texttt{d6\_plane.toml} is the following:

\begin{lstlisting}[language=TOML]
	symmetry_name = "d6_plane"
	# Number of bodies, dimension and masses
	NOB = 3									
	dim = 2
	m = [1, 1, 1]

	# Action type: 0 = Cyclic, 1 = Dihedral, 2 = Brake
	action_type = 1			

	# Group generators
	kern = "TrivialKerTau(2)"
	rotV = "[[1, 0], [0, 1] ]"
	rotS = "(1,2,3)"
	refV = "[[-1,  0], [ 0, -1] ]"
	refS = "(1,2)"

	# Other configs
	F = 24
	Omega = [
		[0, 0],
		[0, 0]
	]
\end{lstlisting}

\noindent Open the Julia REPL and import the package \texttt{SymOrb.jl}:

\begin{lstlisting}[language=Julia]
	using SymOrb
\end{lstlisting}
Then, initialise the problem to get an object of type \texttt{Problem} that contains all the information about the group, the masses, the number of Fourier coefficients and the potential function.
\begin{lstlisting}[language=Julia]
	P = initialize("d6_plane.toml")
\end{lstlisting}

Next, run the optimisation routine. The default behaviour is to find a single orbit, starting from a random path, using a BFGS method that stops after 200 steps or when the norm of the action gradient is smaller than $10^{-8}$. To get a list of the possible options, use the command \texttt{?find\_orbits} in the REPL.
\begin{lstlisting}[language=Julia]
	result = find_orbits(P)
\end{lstlisting}
The output is an array of objects of type \texttt{MinimizationResult}, that in this case contains only one element. To retrieve the Fourier coefficients of the found orbit, access the field \texttt{fourier\_coeff} of the object \texttt{result[1]}.

The last step is to visualise the output. To enable the animation, the package \texttt{GLMakie.jl} must be installed and imported. 

\begin{lstlisting}[language=Julia]
	using GLMakie
\end{lstlisting}

The animation routine allows to visualise the complete periodic orbit and requires the information on the symmetry contained in \texttt{P} and the Fourier coefficients on the fundamental domain,  \texttt{result[1].fourier\_coeff}

\begin{lstlisting}[language=Julia]
	path_animation(P, result[1].fourier_coeff)
\end{lstlisting}

\begin{figure}
	\centering
	\includegraphics[width=0.8\textwidth]{"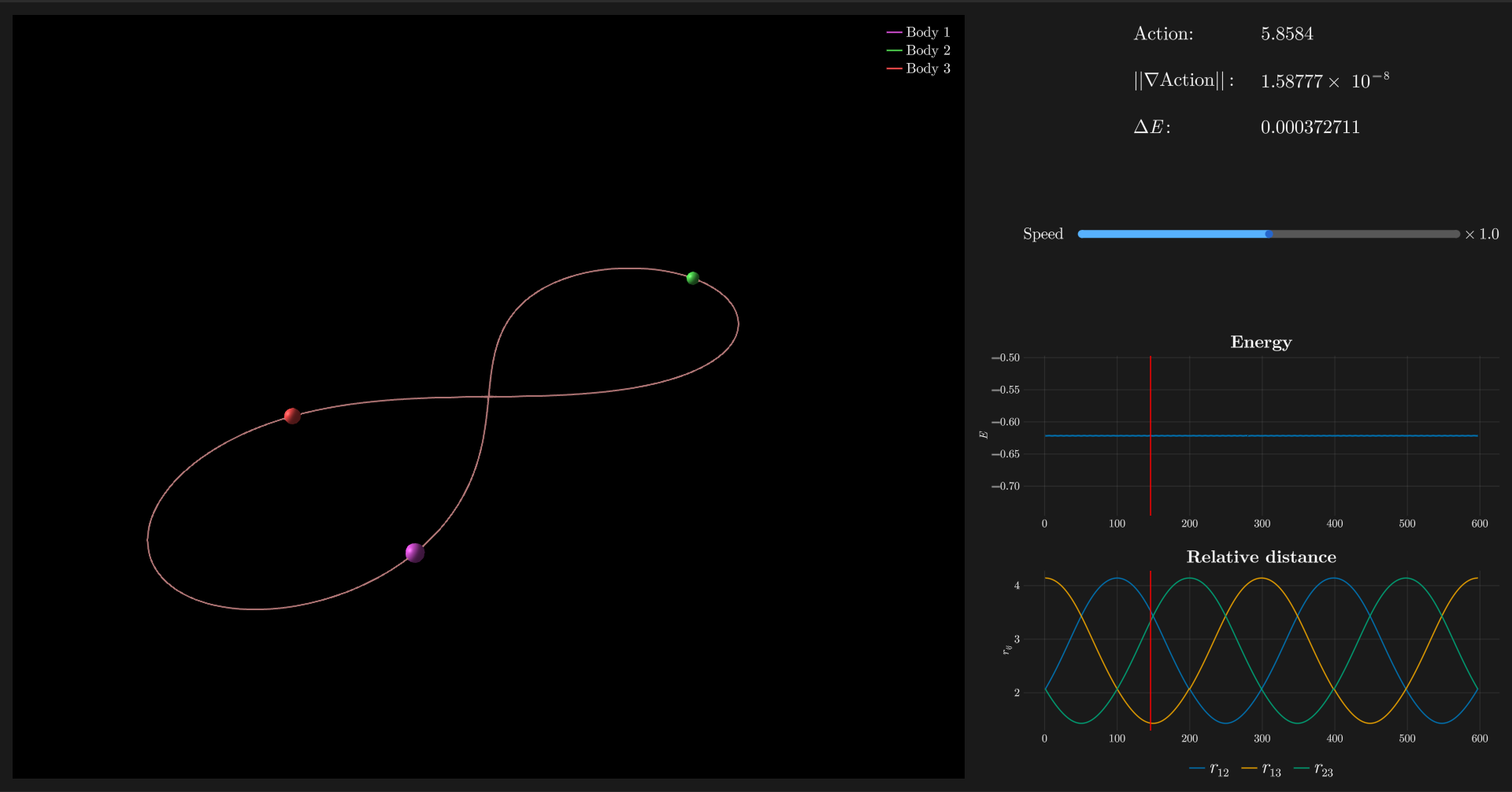"}
	\caption{The visualisation periodic orbit found by the optimisation routine described in section \ref{sec:example_session}.}
\end{figure}

In order to keep the obtained result, a copy of the output of the optimisation routine will be written in a \texttt{TOML} file whose name is the same as the action value, in a directory named \texttt{d6\_plane}. The following command will read the file (in this example, \texttt{"d6\_plane/5.8584.toml"}) and returns the problem configuration \texttt{P} and the Fourier coefficients \texttt{coeff}. 
\begin{lstlisting}[language=Julia]
	P, coeff = read_path_from_file("d6_plane/5.8584.toml")
\end{lstlisting}

\begin{remark}\label{rem:com3}
	Since the optimised Fourier coefficients refer only to the first $n-1$ bodies, as already stated in Remarks \ref{rem:com} and \ref{rem:com2}, the dimension of output \texttt{coeff} will be $(F+2)\times (n-1) \times d$, in this case $26\times 2\times 2 = 104$. It is though possible to obtain the actual trajectories on the whole period, including the $n$-th body's one, using the functions \texttt{extend\_to\_period} and \texttt{build\_path} as follows:

\begin{lstlisting}[language=Julia]
	y = build_path(P, x)
	x = extend_to_period(P, coeff)
\end{lstlisting}
The $k-$th component of the position of the $i$-th body at time step $h$ is stored in \texttt{y[k, i, h]}.
\end{remark}

\subsection*{Acknowledgements}

We thank Irene De Blasi for usefully commenting on the original version of the code \cite{symorb}. 

\noindent The first, third, fifth and sixth authors are partially supported by INDAM-GNAMPA research group. The first author is partially supported by the PRIN 2022 project 2022FPZEES -- \emph{Stability in Hamiltonian Dynamics and Beyond}. The fifth and sixth author are partially supported by the PRIN 2022 project 20227HX33Z -- \emph{Pattern formation in nonlinear phenomena.}

\appendix

\section{Finite subgroups of $O(2)$}

Our major interest resides on some particular finite groups, which turn out to be very useful to describe the symmetries acting on the $n$-body problem. In Section \ref{sec:sym_group} we have recalled the concept of \emph{group action}, which is a very smart way of describing the symmetry of an object and allows to store in a group all the necessary information to describe these symmetries. There are some particular classes of groups whose action models some reductions by symmetry of the $n$-body problem and we collect them in this short section.

\begin{definition}[Dihedral groups]\label{def:dihedral}
	A \emph{dihedral group} is the group of symmetries of a regular polygon, i.e., the group of \emph{reflections} and \emph{rotations} which fix the polygon. If we consider a polygon with $n$ edges, also named $n$-\emph{gon}, we define $D_{2n}$ as its dihedral group. The elements of $D_{2n}$ will be rotational symmetries, reflectional symmetries and their compositions. Note that rotations which fix a polygon with $n$ edges are those of $\frac{2\pi}{n}$ and their multiples. The order of the dihedral group $D_{2n}$ is exactly $2n$, since it contains $n$ rotational symmetries and $n$ reflectional symmetries.  The dihedral group $D_{2n}$ can be generated by a unique rotation $r$ and a unique reflection $s$, so that it is usually defined through the following presentation:
	\[
	\langle s,r:\,s^2=r^n=(sr)^2=1\rangle.
	\]
\end{definition}

It is quite clear that the 2-dimensional \emph{orthogonal group} $O(2)$ plays a foundamental role in the context of this paper. We briefly recall that $O(2)$ is the matrix group of $2\times 2$ real orthogonal matrices, and it contains the \emph{special orthogonal group} $SO(2)=\left\lbrace M\in O(2):\,\det (M)=1\right\rbrace$. Therefore, the elements act either as a rotation about the origin or a reflection with respect to a line passing through the origin. For instance, we can represent a rotation of angle $\vt$ in this way:
\[
\begin{pmatrix}
	\cos\vt & -\sin\vt \\
	\sin\vt & \cos\vt
\end{pmatrix}
\]
and a reflection with respect to a line which forms an angle $\vt$ with the horizontal axis in this way:
\[
\begin{pmatrix}
	\cos2\vt & \sin2\vt \\
	\sin2\vt & -\cos2\vt
\end{pmatrix}.
\]
Thus, we can expect that the classification of the finite subgroups of $O(2)$ is quite precise. We have indeed the following classical result. 

\begin{proposition}[Finite subgroups of $O(2)$]\label{prop:subgroups_orthogonal}
	Any finite subgroup $H$ of $O(2)$ is either cyclic or dihedral. In particular:
	\begin{itemize}
		\item if $H\sset SO(2)$ then $H$ is cyclic and only contains rotations;
		\item if $H=\{1,S\}$, with $S\in O(2)\setminus SO(2)$ a reflection, then $H$ is cyclic of order 2;
		\item if $H$ contains both elements of $SO(2)$ and $O(2)\setminus SO(2)$ then it is dihedral.
	\end{itemize}
\end{proposition}

\section{Absence of collisions for local minimisers}\label{app:rcp}

As announced in Section \ref{sec:optimisation}, we collect here a survey on the arguments which guarantee the absence of collisions for local minimisers of the action functional $\mathcal{A}$ (we mainly refer to Sections 9-10 of \cite{FT2003}). For a finite group $G$, we recall that $\mathbb{I}$ is the fundamental of the action of $G$ (see Definition \ref{def:fundamental_domain}) and the restriction of $\mathcal{A}$ to $\mathbb{I}$ is the functional $\mathcal{A}_{\mathbb{I}}\colon Y\to\R\cup\{+\infty\}$ which reads
\[
\mathcal{A}_{\mathbb{I}}(y)=\int_{\mathbb{I}}\left[\frac12\sum\limits_{i=1}^n m_i\|\dot{y}(t)\|^2+\sum\limits_{i<j}\frac{m_im_j}{\|y_i(t)-y_j(t)\|}\right]\,dt,
\]
where $Y$ is defined in Proposition \ref{prop:homo_fund_domain}. Lemma \ref{lem:local_coercivity} inspires the search for minimisers for $\mathcal{A}_{\mathbb{I}}$ in $Y$. The following definition specifies the nature of minimisers in this context, with respect to local deformations in $Y$. 
 
\begin{definition}
	We say that $y\in Y$ is a \emph{local minimiser} of $\mathcal{A}_{\mathbb{I}}$ if, for any $\vp\in Y$ and for any $\ve$ sufficiently small we have
	\[
	\mathcal{A}_{\mathbb{I}}(x)\leq\mathcal{A}_{\mathbb{I}}(x+\ve\vp).
	\]
\end{definition}
Assuming $\mathcal{X}^G=\{0\}$, the existence of at least a local minimiser is guaranteed by Lemma \ref{lem:local_coercivity} (see also \cite[Proposition 4.12]{FT2003}), which once symmetrised is a minimiser of $\mathcal{A}$ in $\Lambda^G$, defined on the whole period. Since $\Lambda^G$ is closed, the minimiser could lie in the boundary and collisions between the bodies may occur. The following definitions prepare the reader to a stronger condition on $G$ -- known as the \emph{rotating circle property} -- which forces the minimisers to avoid collisions. 
\begin{definition}
	Let $G$ be a finite group acting orthogonally on $\R^d$ through the representation $\rho$ and $\mathbb{S}\sset\R^d$ be a circle centred at the origin. We say that $\mathbb{S}$ is \emph{rotating under a subgroup} $H$ of $G$ if the restriction $h|_\mathbb{S}$ of the action of every element $h\in H$ is a rotation of an angle $\vt_h$. In particular, in this case $\mathbb{S}$ is invariant under the action of $H$.
\end{definition}

\begin{definition}
	Let $G$ be a finite group acting on the index set $\{1,\ldots,n\}$ through the representation $\sigma$. Fix $i\in\{1,\ldots,n\}$ and let $H$ be a subgroup of $G$, let $K_i$ denote the isotropy subgroup of the index $i$ in $H$, i.e., 
	\[
	K_i\uguale\{h\in K:\,hi=i\}.
	\] 
	A circle $\mathbb{S}\sset\R^d$ centred at the origin is termed \emph{rotating for i under} $H$ if
	\begin{itemize}
		\item $\mathbb{S}$ is rotating under $H$;
		\item $S\sset(\R^d)^{K_i}$.
	\end{itemize}
\end{definition}

\begin{definition}
	We say that a group $G$ acts on $\Lambda$ with the rotating circle property if, for every $\mathbb{T}$-isotropy subgroup $G_t$ and for at least $n-1$ indexes $i\in\{1,\ldots,n\}$, there exists a circle $\mathbb{S}\sset\R^d$ centred at the origin rotating for $i$ under $G_t$.
\end{definition}

The following result establishes in two steps the absence of collisions for local minimisers.
\begin{lemma}[Theorem 10.3, \cite{FT2003}]\label{lem:ferrario_terracini}
	Assume that $\ker\tau$ acts on $\Lambda$ with the rotating circle property. Then, any local minimiser $y\in Y$ of $\mathcal{A}_{\mathbb{I}}$ is free of interior collisions, i.e., $y(t)\not\in\Delta$, for any $t\in(t_0,t_1)$. Moreover, if also $H_0$ and $H_1$ act on $\Lambda$ with the rotating circle property, $y(t_0),y(t_1)\not\in\Delta$ and the minimiser is always collision-less. 
\end{lemma}
Note that, in the previous lemma, we have made no distinction between the cyclic and not cyclic action cases. The reason is again that, when $\bar{G}$ is cyclic, then $H_0=H_1=\ker\tau$ and the absence of collisions in the whole period follows from the fact that $\ker\tau$ acts with the rotating circle property. Moreover, note that if $\ker\tau=1$, the rotating circle property holds trivially, so that any cyclic group $G$ generates $G$-equivariant collision-less solutions as action-minimisers.

To summarise, the main result proved in \cite{FT2003} is the following

\begin{theorem}
	Assume that $\mathcal{X}^G=\{0\}$ and that $\ker\tau$, $H_0$ and $H_1$ act on $\Lambda$ with the rotating circle property. Then, there exists a classical solution of the $G$-equivariant $n$-body problem. 
\end{theorem}

\bibliography{references}
\bibliographystyle{plain}

\medskip

\noindent
V. Barutello\\
Dipartimento di Matematica ``Giuseppe Peano'', Universit\`a degli Studi di Torino\\
Via Carlo Alberto 10, 10123 Torino, Italy\\
\texttt{vivina.barutello@unito.it}

\medskip

\noindent
M. G. Bergomi\\
Independent Researcher \\
Milano, Italy\\
\texttt{mattiagbergomi@gmail.com}

\medskip

\noindent
G. M. Canneori \\
Dipartimento di Matematica ``Giuseppe Peano'', Universit\`a degli Studi di Torino\\
Via Carlo Alberto 10, 10123 Torino, Italy\\
\texttt{gianmarco.canneori@unito.it}

\medskip

\noindent 
R. Ciccarelli \\
Dipartimento di Matematica ``Giuseppe Peano'', Universit\`a degli Studi di Torino\\
Via Carlo Alberto 10, 10123 Torino, Italy\\
\texttt{roberto.ciccarelli@unito.it} 

\medskip

\noindent
D. L. Ferrario \\
Dipartimento di Matematica e Applicazioni, Universit\`a degli Studi di Milano-Bicocca\\
Via Roberto Cozzi 55, 20125 Milano, Italy\\
\texttt{davide.ferrario@unimib.it}

\medskip

\noindent
S. Terracini\\
Dipartimento di Matematica ``Giuseppe Peano'', Universit\`a degli Studi di Torino\\
Via Carlo Alberto 10, 10123 Torino, Italy\\
\texttt{susanna.terracini@unito.it}

\medskip

\noindent
P. Vertechi \\
Independent Researcher \\
Trieste, Italy\\
\texttt{pietro.vertechi@protonmail.com}

\end{document}